\documentclass[12pt]{amsart}
\usepackage{amscd,amsmath,amssymb,enumerate,amsfonts,mathrsfs,txfonts}
\usepackage{comment}
\usepackage{hyperref}
\usepackage{xcolor}
\usepackage{tikz,pgfplots,pgf}
\usepackage{vmargin}
\usepackage[all]{xy}

\tikzset{node distance=3cm, auto}
\usetikzlibrary{calc} 
\usetikzlibrary{trees} 

\setpapersize{A4}
\setmargins{2.5cm}      
{1.5cm}                        
{16.5cm}                      
{23.42cm}                   
{10pt}                          
{1cm}                           
{0pt}                          
{2cm}

\newtheorem{theorem}{Theorem}[section]
\newtheorem{lemma}[theorem]{Lemma}
\newtheorem{proposition}[theorem]{Proposition}
\newtheorem{corollary}[theorem]{Corollary}
\theoremstyle{definition}
\newtheorem{definition}[theorem]{Definition}

\numberwithin{equation}{section}

\def\N{\mathbb{N}}
\def\R{\mathbb{R}}
\def\C{\mathbb{C}}
\def\Cu{\mathscr{C}}

\def\F{\mathscr{F}}
\def\G{\mathcal{G}}
\def\H{\mathcal{H}}

\def\L{\mathcal{L}}
\def\D{\mathbb{D}}
\def\B{\mathcal{B}}

\newcommand{\innerproduct}[2]{\langle #1 \, | \, #2 \rangle}

\def\lin{\mathrm{lin}}

\makeatletter
\@namedef{subjclassname@2020}{\textup{2020} Mathematics Subject Classification}
\makeatother

\begin{document}

\title[Factorization of Bloch mappings through a Hilbert space]{Factorization of Bloch mappings through a Hilbert space}

\author[M. G. Cabrera-Padilla]{M. G. Cabrera-Padilla}
\address[M. G. Cabrera-Padilla]{Departamento de Matem\'aticas, Universidad de Almer\'ia, Ctra. de Sacramento s/n, 04120 La Ca\~{n}ada de San Urbano, Almer\'ia, Spain.}
\email{m\_gador@hotmail.com}

\author[A. Jim{\'e}nez-Vargas]{A. Jim\'enez-Vargas}
\address[A. Jim{\'e}nez-Vargas]{Departamento de Matem\'aticas, Universidad de Almer\'ia, Ctra. de Sacramento s/n, 04120 La Ca\~{n}ada de San Urbano, Almer\'ia, Spain.}
\email{ajimenez@ual.es}

\author[D. Ruiz-Casternado]{D. Ruiz-Casternado}
\address[D. Ruiz-Casternado]{Departamento de Matem\'aticas, Universidad de Almer\'ia, Ctra. de Sacramento s/n, 04120 La Ca\~{n}ada de San Urbano, Almer\'ia, Spain.}
\email{drc446@ual.es}

\date{\today}

%30H30 Bloch spaces
%46E15 Banach spaces of continuous, differentiable or analytic functions
%46E40 Spaces of vector- and operator-valued functions
%47B38 Linear operators on function spaces (general)


%\thanks{Corresponding author: A. Jim\'enez-Vargas (ajimenez@ual.es)}
\thanks{Research partially supported by Ministerio de Ciencia e Innovaci\'{o}n grant PID2021-122126NB-C31 funded by MICIU/AEI/10.13039/501100011033 and by ERDF/EU, and by Junta de Andaluc\'ia grant FQM194.}

\begin{abstract}
We introduce the concept of vector-valued holomorphic mappings on the complex unit disc that factor through a Hilbert space and state the main properties of the space formed by such Bloch mappings equipped with a natural norm: linearization, Bloch transposition, surjective and injective Banach ideal property, Kwapie\'n-type characterization by Bloch domination, and duality.
\end{abstract}
\maketitle

%%%%%%%%%%%%%%%%%%%%%%%%%%%%%%%%%%%%%%%%%%%%%%%%%%%%%%%%%%%%%%%%%%%%%%%%%%%%%%%%%%%%%%%%%%%%%%%%%%%%%%%%%%%%%%%%%%%%%%%%%%%%%%%%%%%%%%%%%%%%%%%%%%%%%%%%%%%%%%%%%%%%%%%%%%%%%%%%%%

\section*{Introduction}\label{section 0}

Bounded linear operators between Banach spaces that factor through a Hilbert space appeared in Grothendieck's R\'esume \cite{Gro-53}, with the name of $h$-integral operators. Their study was encouraged by Lindenstrauss and Pe\l czy\'{n}ski \cite{LinPel-68}, and their duality theory was addressed by Kwapie\'n \cite{Kwa-72}. Basic results on this class of operators can be consulted in the monographs by Diestel, Jarchow and Tonge \cite[Chapter 7]{DisJarTon-95}, Li and Queff\'{e}lec \cite[Chapter 5]{LiQue-18} and Pisier \cite[Chapter 2]{Pis-86}. 

The extension of this theory to different nonlinear settings has been promoted by some authors. This is the case of completely bounded operators between operator spaces by Pisier \cite{Pis-96}, Lipschitz mappings between metric spaces by Ch\'avez-Dom\'inguez \cite{Cha-14}, multilinear operators between Banach spaces by Fern\'andez-Unzueta and Garc\'ia-Hern\'andez \cite{FerGar-19}, and $2$-dominated polynomials by Rueda and S\'anchez P\'erez \cite{RueSan-14}.

We will study in this note the factorization through a Hilbert space of vector-valued Bloch mappings on the complex unit open disc $\D$. 

Let $X$ be a complex Banach space and let $\H(\D,X)$ be the space of all holomorphic mappings from $\D$ into $X$. A mapping $f\in\H(\D,X)$ is said to be Bloch if there is a constant $c\geq 0$ such that $(1-|z|^2)\left\|f'(z)\right\|\leq c$ for all $z\in\D$. The normalized Bloch space $\widehat{\B}(\D,X)$ is the Banach space of all mappings $f\in\H(\D,X)$ with $f(0)=0$ so that   
$$
\rho_{\B}(f):=\sup\left\{(1-|z|^2)\left\|f'(z)\right\|\colon z\in\D\right\}<\infty ,
$$ 
equipped with the Bloch norm $\rho_{\B}$. To simplify, we will write $\widehat{\B}(\D)$ instead of $\widehat{\B}(\D,\C)$. Spaces of Bloch functions have been studied, among others, by Anderson, Clunie and Pommerenke \cite{AndCluPom-74} and Zhu \cite{Zhu-07} in the complex-valued case, and by Arregui and Blasco \cite{ArrBlas-03,Bla-90} in the vector-valued case.

Given two Banach spaces $X$ and $Y$, we denote by $\mathcal{L}(X,Y)$ the Banach space of all bounded linear operators from $X$ into $Y$, equipped with the operator canonical norm. As usual, $X^*$ stands for $\mathcal{L}(X,\mathbb{C})$.  
%For $x\in X$ and $x^*\in X^*$, we will sometimes write $\langle x^*,x\rangle=x^*(x)$. Moreover, $B_X$ and $S_X$ stand for the closed unit ball of $X$ and the unit sphere of $X$, respectively.  

A mapping $f\in\H(\D,X)$ is said to factor through a Hilbert space if there exists a Hilbert space $H$, a mapping $g\in\widehat{\B}(\D,H)$ and an operator $T\in\L(H,X)$ such that $f=T\circ g$, that is, the following diagram commutes:
\begin{equation*}
	\xymatrix{
	 &H  \ar[dr]^{T} &  \\
	\D \ar[ur]^{g} \ar[rr]_{f}  & &X  
	}
\end{equation*}
For such a $f$, we set $\gamma^{\B}_2(f)=\inf\left\{\left\|T\right\|\rho_\B(g)\right\}$, where the infimum runs over all possible factorizations of $f$ as above. We denote by $\Gamma^{\widehat{\B}}_2(\D,X)$ the linear space of all mappings $f\in\H(\D,X)$ for which $f$ admits such a factorization. 

The key tool to obtain our results is a linearization process of Bloch mappings on $\D$ by using a predual space of $\widehat{\B}(\D)$, called Bloch-free space over $\D$ and denoted by $\G(\D)$. We will devote Section \ref{1} on preliminaries to recall this process that was introduced in \cite{JimRui-22} and applied to study $p$-summing Bloch mappings on $\D$ in \cite{CabJimRui-23}. Once this is done, we will present the most important properties of $\Gamma^{\widehat{\B}}_2$-spaces in five sections. 

In Section \ref{2}, we characterize the factorization through a Hilbert space of a zero-preserving Bloch mapping from $\D$ into $X$ by means of the linear factorization through a Hilbert space of its linearization from $\G(\D)$ into $X$ as well as of its Bloch transposition from $X^*$ into $\widehat{\B}(\D)$. 

Using this characterization, we prove in Section \ref{3} that $[\Gamma^{\widehat{\B}}_2,\gamma^{\B}_2]$ is a surjective and injective Banach ideal of normalized Bloch mappings. In \cite{CabJimRui-23}, the concept of $p$-summing Bloch mapping from $\D$ into $X$ for $1\leq p\leq\infty$ was introduced and studied. We show that every $2$-summing Bloch mapping $f\colon\D\to X$ with $f(0)=0$ factors through a Hilbert space. 

Section \ref{4} contains the main result of this paper which states a Bloch version of a known theorem of Kwapie\'n \cite{Kwa-72} by characterizing holomorphic mappings that factor through a Hilbert space in terms of a Bloch subordination/domination criterion of finite sequences of elements of $\C\times\D$. 

Section \ref{5} presents a convenient norm on the space of the so-called $X$-valued Bloch molecules on $\D$ so that the dual of the completion of this last space is identified with $(\Gamma^{\widehat{\B}}_2(\D,X^*),\gamma^{\B}_2)$.  

%%%%%%%%%%%%%%%%%%%%%%%%%%%%%%%%%%%%%%%%%%%%%%%%%%%%%%%%%%%%%%%%%%%%%%%%%%%%%%%%%%%%%%%%%%%%%%%%%%%%%%%%%%%%%%%%%%%%%%%%%%%%%%%%%%%%%%%%%%%%%%%%%%%%%%%%%%%%%%%%%%%%%%%%%%%%%%%%%%

\section{Preliminaries}\label{1}

We first collect some useful results on the linearization of Bloch mappings borrowed from \cite{JimRui-22}. For each $z\in\D$, a Bloch atom of $\D$ is a bounded linear functional $\gamma_z\colon\widehat{\B}(\D)\to\mathbb{C}$ given by 
$$
\gamma_z(f)=f'(z)\qquad (f\in\widehat{\B}(\D)).
$$
The elements of $\lin(\{\gamma_z\colon z\in\D\})$ in $\widehat{\B}(\D)^*$ are called Bloch molecules of $\D$. There exist other standard designations in the literature for Bloch atoms and Bloch molecules as, for example, point evaluations and characters or states, respectively. The Bloch-free Banach space over $\D$, denoted by $\G(\D)$, is the norm-closed linear hull of $\left\{\gamma_z\colon z\in\D\right\}$ in $\widehat{\B}(\D)^*$. As usual, for any $T\in\L(X,Y)$, $T^*\colon Y^*\to X^*$ denotes the adjoint operator of $T$. 

\begin{theorem}\cite{JimRui-22}\label{main-theo}
\begin{enumerate}
\item For every $z\in\D$, the function $f_z\colon\D\to\C$ defined by 
$$
f_z(w)=\frac{(1-|z|^2)w}{1-\overline{z}w}\qquad (w\in\D),
$$ 
belongs to $\widehat{\B}(\D)$ with $\rho_{\B}(f_z)=1=(1-|z|^2)f_z'(z)$.
\item The mapping $\Gamma\colon\D\to\G(\D)$, defined by $\Gamma(z)=\gamma_z$ for all $z\in\D$, is holomorphic with $\left\|\gamma_z\right\|=1/(1-|z|^2)$ for all $z\in\D$. %The functional $\widehat{\gamma}_z:=(1-|z|^2)\gamma_z$ is referred to as a \textit{normalized Bloch atom of} $\D$.% and $\Gamma'(z)(f)=f''(z)$ for all $f\in\widehat{\B}(\D)$ and $z\in\D$.
%\item $\Gamma(\D)$ is a linearly independent subset of $\G(\D)$.
\item $\widehat{\B}(\D)$ is isometrically isomorphic to $\G(\D)^*$, via the evaluation mapping $\Lambda\colon\widehat{\B}(\D)\to\G(\D)^*$ given by
$$
\Lambda(f)(\gamma)=\sum_{k=1}^n\lambda_kf'(z_k)\qquad (f\in\widehat{\B}(\D),\; \gamma=\sum_{k=1}^n\lambda_k\gamma_{z_k}\in\lin(\Gamma(\D))).
$$
%Its inverse is the mapping $\Lambda^{-1}\colon\G(\D)^*\to\widehat{\B}(\D)$ defined by 
%$$
%\Lambda^{-1}(\phi)(z)=\int_{[0,z]}\phi(\gamma_w)\ dw \qquad \left(\phi\in\G(\D)^*,\; z\in\D\right).
%$$
%\item Let $\gamma\in\G(\D)$. Then, for every $\varepsilon>0$, there exist sequences $(\lambda_n)_{n\geq 1}\in\ell_1$ with $\sum_{n=1}^\infty\left|\lambda_n\right|<\left\|\gamma\right\|+\varepsilon$ and $(z_n)_{n\geq 1}\in\D^{\mathbb{N}}$ such that $\gamma=\sum_{n=1}^\infty \lambda_n\widehat{\gamma}_{z_n}$. Moreover,
%$$
%\left\|\gamma\right\|=\inf\left\{\sum_{n=1}^\infty\left|\lambda_n\right|\colon \gamma=\sum_{n=1}^\infty \lambda_n\widehat{\gamma}_{z_n},\ (\lambda_n)\in\ell_1,\; (z_n)\in\D^{\mathbb{N}}\right\}.
%$$
%\item $B_{\G(\D)}$ coincides with $\overline{\abco}(\mathcal{M}(\D))\subseteq\widehat{\B}(\D)^*$, where $\mathcal{M}(\D):=\{\widehat{\gamma}_z\colon z\in\D\}$.
%\item Let $(f_i)$ be a net in $\widehat{\B}(\D)$ and $f\in\widehat{\B}(\D)$.
%\begin{enumerate}
%	\item If $(f_i)\to f$ weak* in $\widehat{\B}(\D)$, then $(f_i)\to f$ pointwise on $\D$.
%	\item If $(f_i)$ is bounded in $\widehat{\B}(\D)$ and $(f_i)\to f$ pointwise on $\D$, then $(f_i)\to f$ weak* in $\widehat{\B}(\D)$.
%\end{enumerate}
%\item $\widehat{\B}_0(\D)$ is weak*-dense in $\widehat{\B}(\D)$. 
%\item $\G(\D)$ is isometrically isomorphic to $\widehat{\B}_0(\D)^*$, via the restriction mapping $R\colon\G(\D)\to\widehat{\B}_0(\D)^*$ defined by
%$$
%R(\gamma)(f)=\gamma(f)\qquad (f\in\widehat{\B}_0(\D),\; \gamma\in\G(\D)).
%$$
\item For every holomorphic function $h\colon\D\to\D$ with $h(0)=0$, the composition operator $C_h\colon\widehat{\B}(\D)\to\widehat{\B}(\D)$, defined by $C_h(f)=f\circ h$, is in  $\L(\widehat{\B}(\D),\widehat{\B}(\D))$ with $\rho_\B(h)\leq||C_h||\leq 1$. 
\item For every holomorphic function $h\colon\D\to\D$ such that $h(0)=0$, there exists a unique operator $\widehat{h}\in\L(\G(\D),\G(\D))$ such that $\widehat{h}\circ\Gamma=h'\cdot(\Gamma\circ h)$. In fact, $\left(\widehat{h}\right)^*\equiv C_h$.% and thus $||\widehat{h}||=||C_h||$.  
\item For every complex Banach space $X$ and every mapping $f\in\widehat{\B}(\D,X)$, there exists a unique operator $S_f\in\L(\G(\D),X)$ such that $S_f\circ\Gamma=f'$. Furthermore, $||S_f||=\rho_{\B}(f)$.  
%, that is, the diagram 
%$$
%\begin{tikzpicture}
%\node (D) {$\D$};
%\node (GD) [below of=D] {$\G(\D)$};
%\node (X) [right of=GD] {$X$};
%\draw[->] (D) to node {$f'$} (X);
%\draw[->] (D) to node [swap] {$\Gamma$} (GD);
%\draw[->, dashed] (GD) to node [swap] {$S_f$} (X);
%\end{tikzpicture}
%$$
%commutes. 
%\end{comment}
\item The correspondence $f\mapsto S_f$ is an isometric isomorphism from $\widehat{\B}(\D,X)$ onto $\L(\G(\D),X)$.
\item Given $f\in\widehat{\B}(\D,X)$, the map $f^t\colon X^*\to \widehat{\B}(\D)$, defined by $f^t(x^*)=x^*\circ f$ and called Bloch transpose of $f$, is in $\L(X^*,\widehat{\B}(\D))$ with $\|f^t\|=\rho_\B(f)$. Furthermore, $f^t=\Lambda^{-1}\circ (S_f)^*$. $\hfill\qed$ %, where $(S_f)^*\colon X^*\to\G(\D)^*$ is the adjoint operator of $S_f$. 
\end{enumerate}
\end{theorem}

We will now present the most important properties of vector-valued Bloch mappings on $\D$ that factor through a Hilbert space.

%%%%%%%%%%%%%%%%%%%%%%%%%%%%%%%%%%%%%%%%%%%%%%%%%%%%%%%%%%%%%%%%%%%%%%%%%%%%%%%%%%%%%%%%%%%%%%%%%%%%%%%%%%%%%%%%%%%%%%%%%%%%%%%%%%%%%%%%%%%%%%%%%%%%%%%%%%%%%%%%%%%%%%%%%%%%%%%%%%

\section{Linearization / Bloch transposition}\label{2}

We study the relationships of a map $f\in\Gamma^{\widehat{\B}}_2(\D,X)$ with both its linearization $S_f\in\L(\G(\D),X)$ and its Bloch transposition $f^t\in\L(X^*,\widehat{\B}(\D))$. Compare with \cite[Proposition 7.2]{DisJarTon-95}. 

Let us recall that given Banach spaces $X$ and $Y$, an operator $T\in\L(X,Y)$ is said to factor through a Hilbert space if there exists a Hilbert space $H$ and operators $B\in\L(X,H)$ and $A\in\L(H,Y)$ such that $T=A\circ B$. The set of all operators $T\in\L(X,Y)$ admitting  such a factorization is denoted by $\Gamma_2(X,Y)$, and it becomes a Banach space equipped with the norm $\gamma_2(T)=\inf\left\{\left\|A\right\|\left\|B\right\|\right\}$, where the infimum extends over all possible factorizations of $T$ as above. 

We begin with the following easy observation.

\begin{lemma}\label{lem-1}
Let $X$ be a complex Banach space. Then $\Gamma^{\widehat{\B}}_2(\D,X)\subseteq\widehat{\B}(\D,X)$ with $\rho_{\B}(f)\leq\gamma_2^{\B}(f)$ for all $f\in\Gamma^{\widehat{\B}}_2(\D,X)$.
\end{lemma}

\begin{proof}
If $f\in\Gamma^{\widehat{\B}}_2(\D,X)$, then $f=T\circ g$, where $g\in\widehat{\B}(\D,H)$ and $T\in\L(H,X)$ for some Hilbert space $H$. We have 
$$
(1-|z|^2)\left\|f'(z)\right\|=(1-|z|^2)\left\|T(g'(z))\right\|\leq \left\|T\right\|(1-|z|^2)\left\|g'(z)\right\|\leq \left\|T\right\|\rho_{\B}(g)
$$
for all $z\in\D$, and therefore $f\in\widehat{\B}(\D,X)$ with $\rho_{\B}(f)\leq\left\|T\right\|\rho_{\B}(g)$. Taking the infimum over all such factorizations of $f$, we obtain $\rho_{\B}(f)\leq\gamma_2^{\B}(f)$.
\end{proof}

\begin{theorem}\label{theo-1}
Let $X$ be a complex Banach space and $f\in\widehat{\B}(\D,X)$. The following statements are equivalent:
\begin{enumerate}
\item $f\colon\D\to X$ factors through a Hilbert space.
\item $S_f\colon\G(\D)\to X$ factors through a Hilbert space.
\item $f^t\colon X^*\to\widehat{\B}(\D)$ factors through a Hilbert space.
\end{enumerate}
In this case, $\gamma^{\B}_2(f)=\gamma_2(S_f)=\gamma_2(f^t)$. Furthermore, the mapping $f\mapsto S_{f}$ is an isometric isomorphism from $\Gamma^{\widehat{\B}}_2(\D,X)$ onto $\Gamma_2(\G(\D),X)$.
\end{theorem}

\begin{proof}
$(i)\Rightarrow (ii)$: If $f\in\Gamma^{\widehat{\B}}_2(\D,X)$, then $f=T\circ g$, where $g\in\widehat{\B}(\D,H)$ and $T\in\L(H,X)$ for some Hilbert space $H$. By Theorem \ref{main-theo}, we have the factorizations $f'=S_f\circ\Gamma$ and $g'=S_g\circ\Gamma$, where $S_f\in\L(\G(\D),X)$ and $S_g\in\L(\G(\D),H)$. Hence $S_f\circ\Gamma=T\circ S_g\circ\Gamma$ and this implies that $S_f=T\circ S_g$ since $\Gamma(\D)$ is a dense linear subspace of $\G(\D)$. Therefore $S_f\in\Gamma_2(\G(\D),X)$ with $\gamma_2(S_f)\leq\left\|T\right\|\left\|S_g\right\|=\left\|T\right\|\rho_\B(g)$. Taking the infimum over all factorizations of $f'$ as above, we have $\gamma_2(S_f)\leq\gamma^{\B}_2(f)$.

$(ii)\Rightarrow (i)$: If $S_f\in\Gamma_2(\G(\D),X)$, then $S_f=T\circ S$, where $S\in\L(\G(\D),H)$ and $T\in\L(H,X)$ for some Hilbert space $H$. Hence $f'=S_f\circ\Gamma=T\circ S\circ\Gamma$ by Theorem \ref{main-theo}. Since $S\circ\Gamma\in\H(\D,H)$, Lemma 2.9 in \cite{JimRui-22} gave us a mapping $h\in\H(\D,H)$ with $h(0)=0$ such that $h'=S\circ\Gamma$. Since 
$$
(1-|z|^2)\left\|h'(z)\right\|=(1-|z|^2)\left\|S(\Gamma(z))\right\|\leq\left\|S\right\|
$$
for all $z\in\D$, it follows that $h\in\widehat{\B}(\D,H)$ with $\rho_{\B}(h)\leq\left\|S\right\|$. Hence $f'=T\circ h'$ and therefore $f\in\Gamma^{\widehat{\B}}_2(\D,X)$ with $\gamma^{\B}_2(f)\leq\left\|T\right\|\left\|S\right\|$. Taking the infimum over all factorizations of $S_f$, we obtain $\gamma^{\B}_2(f)\leq\gamma_2(S_f)$, as desired.

$(ii)\Leftrightarrow (iii)$: Applying first \cite[Proposition 7.2]{DisJarTon-95}, and then Theorem \ref{main-theo} with \cite[Theorem 7.1]{DisJarTon-95}, we get
$$
S_f\in\Gamma_2(\G(\D),X)\Leftrightarrow (S_f)^*\in\Gamma_2(X^*,\G(\D)^*)\Leftrightarrow f^t=\Lambda^{-1}\circ(S_f)^*\in\Gamma_2(X^*,\widehat{\B}(\D)),
$$
and, in this case, $\gamma_2(S_f)=\gamma_2((S_f)^*)=\gamma_2(f^t)$. 

The last assertion of the statement follows immediately by applying that the correspondence $f\mapsto S_f$ from $\widehat{\B}(\D,X)$ into $\L(\G(\D),X)$ is surjective and from what was proved above.
\end{proof}

\section{Ideal property}\label{3}

The concept of a Banach ideal of normalized Bloch mappings was introduced in \cite[Definition 5.11]{JimRui-22}, inspired by the notion of Banach operator ideal due to Pietsch \cite{Pie-80}.

\begin{proposition}\label{prop-10}
$[\Gamma^{\widehat{\B}}_2,\gamma^{\B}_2]$ is a surjective and injective Banach ideal of normalized Bloch mappings. %Se puede mejor [,]
\end{proposition}

\begin{proof}
We are going to show that $[\Gamma^{\widehat{\B}}_2,\gamma^{\B}_2]$ satisfies the properties of the aforementioned definition. Let $X$ be a complex Banach space.

(N1): $(\Gamma^{\widehat{\B}}_2(\D,X),\gamma^{\B}_2)$ is a Banach space with $\gamma^{\B}_2(f)\geq\rho_\B(f)$ for all $f\in\Gamma^{\widehat{\B}}_2(\D,X)$. 

\noindent\emph{Proof:} It follows directly from Lemma \ref{lem-1} and Theorem \ref{theo-1}. \\

(N2): For every $g\in\widehat{\B}(\D)$ and $x\in X$, the mapping $g \cdot x\colon z\mapsto g(z)x$ from $\D$ into $X$ is in $\Gamma^{\widehat{\B}}_2(\D,X)$ with $\gamma^{\B}_2(g\cdot x)=\rho_\B(g)\left\|x\right\|$. 

\noindent\emph{Proof:} %By \cite[Proposition 5.13]{JimRui-22}, we have that 
It is easy to show that $g\cdot x\in\widehat{\B}(\D,X)$ with $(g\cdot x)'=g'\cdot x$ and $\rho_\B(g\cdot x)=\rho_\B(g)\left\|x\right\|$. Since $(g\cdot x)'=(S_g\cdot x)\circ\Gamma$, it follows that $S_{g\cdot x}=S_g\cdot x$ with $\rho_\B(g\cdot x)=\left\|S_{g\cdot x}\right\|=\left\|S_g\cdot x\right\|=\rho_\B(g)\left\|x\right\|$ by Theorem \ref{main-theo}, and therefore $S_{g\cdot x}\in\F(\G(\D),X)$, where $\F(\G(\D),X)$ denotes the subspace of $\L(\G(\D),X)$ formed by all finite-rank operators. Since $\F\subseteq\Gamma_2$, we have $S_{g\cdot x}\in\Gamma_2(\G(\D),X)$ and thus $g\cdot x\in\Gamma^{\widehat{\B}}_2(\D,X)$ with $\gamma^{\B}_2(g\cdot x)=\left\|S_{g\cdot x}\right\|=\rho_\B(g)\left\|x\right\|$ by Theorem \ref{theo-1}. To obtain the reverse inequality, apply (N1).\\

%Another form: We can write $(g\cdot x)'=T_x\circ g'$, where $T_x\in\L(\C,X)$ is defined by $T_x(\lambda)=\lambda x$ for all $\lambda\in\C$. Therefore $g\cdot x\in\Gamma^{\widehat{\B}}_2(\D,X)$ with $\gamma^{\B}_2(g\cdot x)\leq \left\|T_x\right\|\rho_\B(g)=\left\|x\right\|\rho_\B(g)$. For the reverse inequality, apply (N1) and the fact that $\rho_\B(g\cdot x)=\rho_\B(g)\left\|x\right\|$ by \cite[Proposition 5.12]{JimRui-22}.\\

(N3) Ideal property: If $h\colon\D\to\D$ is a holomorphic function with $h(0)=0$, $f\in\Gamma^{\widehat{\B}}_2(\D,X)$ and $T\in\L(X,Y)$ where $Y$ is a complex Banach space, then $T\circ f\circ h\in\Gamma^{\widehat{\B}}_2(\D,Y)$ with $\gamma^{\B}_2(T\circ f\circ h)\leq\|T\|\gamma^{\B}_2(f)$.

\noindent\emph{Proof:} Note that $S_{T\circ f\circ h}=T\circ S_f\circ\widehat{h}\in\Gamma_2(\G(\D),Y)$ by \cite[Proposition 5.13]{JimRui-22}, Theorem \ref{theo-1} and the ideal property of $[\Gamma_2,\gamma_2]$ (see \cite[Proposition 7.1]{DisJarTon-95}). Then, by Theorem \ref{theo-1}, $T\circ f\circ h\in\Gamma^{\widehat{\B}}_2(\D,Y)$ with 
$$
\gamma^{\B}_2(T\circ f\circ h)=\gamma_2(T\circ S_f\circ\widehat{h})\leq \left\|T\right\|\gamma_2(S_f)||\widehat{h}||\leq\left\|T\right\|\gamma^{\B}_2(f).
$$ 

%Another form: We have that $f'=T_0\circ g'$, where $g\in\widehat{\B}(\D,H)$ and $T_0\in\L(H,X)$ for some Hilbert space $H$. Now we can write
%\begin{align*}
%(T\circ f\circ h)'&=T\circ[h'\cdot (f'\circ h)]=T\circ[h'\cdot(T_0\circ g'\circ h)]\\
%                  &=(T\circ T_0)\circ[h'\cdot(g'\circ h)]=(T\circ T_0)\circ(g\circ h)',
%\end{align*}
%where $g\circ h\in\widehat{\B}(\D,H)$ with $\rho_\B(g\circ h)\leq \rho_\B(g)$ by \cite[Proposition 3.6]{JimRui-22}, and $T\circ T_0\in\L(H,Y)$. Therefore $T\circ f\circ h\in\Gamma^{\widehat{\B}}_2(\D,Y)$ with $\gamma^{\B}_2(T\circ f\circ h)\leq\left\|T\right\|\left\|T_0\right\|\rho_\B(g)$. Taking the infimum over all factorizations of $f'$, we deduce that $\gamma^{\B}_2(T\circ f\circ h)\leq\|T\|\gamma^{\B}_2(f)$.\\

(I) Injectivity: For any $f\in\widehat{\B}(\D,X)$, any complex Banach space $Y$ and any metric injection $\iota\in\L(X,Y)$, we have $f\in\Gamma^{\widehat{\B}}_2(\D,X)$ with $\gamma^{\B}_2(f)=\gamma^{\B}_2(\iota\circ f)$ whenever $\iota\circ f\in\Gamma^{\widehat{\B}}_2(\D,Y)$.

\noindent\emph{Proof:} Assume $\iota\circ f\in\Gamma^{\widehat{\B}}_2(\D,Y)$. Then $\iota\circ S_f=S_{\iota\circ f}\in\Gamma_2(\G(\D),Y)$ by \cite[Proposition 5.13]{JimRui-22} and Theorem \ref{theo-1}. Since the Banach operator ideal $[\Gamma_2,\gamma_2]$ is injective by \cite[Proposition 7.3]{DisJarTon-95}, it follows that $S_f\in\Gamma_2(\G(\D),X)$ with $\gamma_2(S_f)=\gamma_2(\iota\circ S_f)$ or, equivalently, $f\in\Gamma^{\widehat{\B}}_2(\D,X)$ with $\gamma^{\B}_2(f)=\gamma^{\B}_2(\iota\circ f)$ again by Theorem \ref{theo-1}.  \\
 
%Another form: Assume that $\iota\circ f\in\Gamma^{\widehat{\B}}_2(\D,Y)$. Then we have $\iota\circ f'=(\iota\circ f)'=T\circ g'$, where $g\in\widehat{\B}(\D,H)$ and $T\in\L(H,Y)$ for some Hilbert space $H$. By the injectivity of $\iota$, there exists $j\in\L(\iota(X),X)$ such that $j\circ\iota=\mathrm{id}_X$. It follows that $f'=(j\circ T)\circ g'$. Hence $f\in\Gamma^{\widehat{\B}}_2(\D,X)$ with $\gamma^{\B}_2(f)\leq \left\|j\circ T\right\|\rho_\B(g)\leq \left\|T\right\|\rho_\B(g)$. Taking the infimum over all factorizations of $f'$, we obtain that $\gamma^{\B}_2(f)\leq \gamma^{\B}_2(\iota\circ f)$. The reverse inequality follows from (N3).\\

(S) Surjectivity: For any $f\in\widehat{\B}(\D,X)$ and any holomorphic function $\pi\colon\D\to\D$ with $\pi(0)=0$ such that $\widehat{\pi}\in\L(\G(\D),\G(\D))$ is a metric surjection, we have $f\in\Gamma^{\widehat{\B}}_2(\D,X)$ with $\gamma^{\B}_2(f)=\gamma^{\B}_2(f\circ\pi)$ whenever $f\circ\pi\in\Gamma^{\widehat{\B}}_2(\D,X)$.

\noindent\emph{Proof:} Assume $f\circ\pi\in\Gamma^{\widehat{\B}}_2(\D,X)$. Since $S_f\circ\widehat{\pi}=S_{f\circ\pi}\in\Gamma_2(\G(\D),X)$ by \cite[Proposition 5.13]{JimRui-22} and Theorem \ref{theo-1}, and the Banach operator ideal $[\Gamma_2,\gamma_2]$ is surjective by \cite[Proposition 7.3]{DisJarTon-95}, we have $S_f\in\Gamma_2(\G(\D),X)$ with $\gamma_2(S_f)=\gamma_2(S_f\circ\widehat{\pi})$, hence $f\in\Gamma^{\widehat{\B}}_2(\D,X)$ with $\gamma^{\B}_2(f)=\gamma_2(S_f)$ and thus $\gamma^{\B}_2(f)=\gamma_2(S_{f\circ\pi})=\gamma^{\B}_2(f\circ\pi)$ by Theorem \ref{theo-1}.
\end{proof}

Let us recall (see \cite{CabJimRui-23}) that a mapping $f\in\H(\D,X)$ is said to be $2$-summing Bloch if there is a constant $c\geq 0$ such that for any $n\in\mathbb{N}$, $\lambda_1,\ldots,\lambda_n\in\mathbb{C}$ and $z_1,\ldots,z_n\in\D$, we have 
$$
\left(\sum_{i=1}^n\left|\lambda_i\right|^2\left\|f'(z_i)\right\|^2\right)^{\frac{1}{2}}\leq c
\sup_{g\in B_{\widehat{\B}(\D)}}\left(\sum_{i=1}^n\left|\lambda_i\right|^2\left|g'(z_i)\right|^2\right)^{\frac{1}{2}}.
$$
The infimum of such constants $c$, denoted by $\pi^{\B}_2(f)$, defines a complete norm on the linear space, denoted by $\Pi^{\widehat{\B}}_2(\D,X)$, consisting of all $2$-summing Bloch mappings $f\colon\D\to X$ for which $f(0)=0$.

If $f\in\Pi^{\widehat{\B}}_2(\D,X)$, Theorem 1.6 in \cite{CabJimRui-23} provides us a regular Borel probability measure $\mu$ on $(B_{\widehat{\B}(\D)},w^*)$, an operator $T\in\L(L_2(\mu),\ell_\infty(B_{X^*}))$ and a mapping $h\in\widehat{\B}(\D,L_\infty(\mu))$ such that
$$
\iota_X\circ f'=T\circ I_{\infty,2}\circ h'\colon \D \stackrel{h'}{\rightarrow}L_\infty(\mu)\stackrel{I_{\infty,2}}{\rightarrow}L_2(\mu)\stackrel{T}{\rightarrow}\ell_{\infty}(B_{X^*}),
$$
where $I_{\infty,2}\colon L_\infty(\mu)\to L_2(\mu)$ is the formal inclusion operator and $\iota_X\colon X\to \ell_\infty(B_{X^*})$ is the isometric linear embedding defined by $\left\langle \iota_X(x),x^*\right\rangle=x^*(x)$. Moreover, $\pi^{\B}_2(f)=\inf\left\{\left\|T\right\|\rho_\B(h)\right\}$, where the infimum is taken over all such factorizations of $\iota_X\circ f'$ as above.%, and this infimum is attained.

Since $(\iota_X\circ f)'=\iota_X\circ f'=T\circ I_{\infty,2}\circ h'=(T\circ I_{\infty,2}\circ h)'$ and $(\iota_X\circ f)(0)=(T\circ I_{\infty,2}\circ h)(0)=0$, we deduce that $\iota_X\circ f=T\circ (I_{\infty,2}\circ h)$. Hence $\iota_X\circ f\in\Gamma_2^{\widehat{\B}}(\D,\ell_{\infty}(B_{X^*}))$ with $\gamma^{\B}_2(\iota_X\circ f)\leq\left\|T\right\|\rho_\B(h)$. By the injectivity of the Banach ideal $[\Gamma^{\widehat{\B}}_2,\gamma^{\B}_2]$, it follows that $f\in\Gamma^{\widehat{\B}}_2(\D,X)$ with $\gamma^{\B}_2(f)=\gamma^{\B}_2(\iota_X\circ f)$. Moreover, $\gamma^{\B}_2(\iota_X\circ f)\leq\pi_2^{\B}(f)$ by taking the infimum over all factorizations of $\iota_X\circ f'$ as above. Thus, we have proved the following. 

\begin{corollary}
Every $2$-summing Bloch mapping $f\colon\D\to X$ with $f(0)=0$ factors through a Hilbert space and $\gamma_2^{\B}(f)\leq\pi_2^{\B}(f)$.$\hfill\qed$
\end{corollary}

%%%%%%%%%%%%%%%%%%%%%%%%%%%%%%%%%%%%%%%%%%%%%%%%%%%%%%%%%%%%%%%%%%%%%%%%%%%%%%%%%%%%%%%%%%%%%%%%%%%%%%%%%%%%%%%%%%%%%%%%%%%%%%%%%%%%%%%%%%%%%%%%%%%%%%%%%%%%%%%%%%%%%%%%%%%%%%%%%%

\section{Bloch domination}\label{4}

We first introduce a Bloch subordination/domination criterion of finite sequences of elements of $\C\times\D$ (compare with \cite[Definition V.3]{LiQue-18}). 

\begin{definition}
%Let $n,m\in\N$, $\lambda_1,\ldots,\lambda_n,\mu_1,\ldots,\mu_m\in\C$ and $z_1,\ldots,z_n,w_1,\ldots,w_m\in\D$. 
%Let $n,m\in\N$, $(\lambda_i)_{i=1}^n\in\C^n$, $(\mu_j)_{j=1}^m\in\C^m$, $(z_i)_{i=1}^n\in\D^n$ and $(w_j)_{j=1}^m\in\D^m$.
Let $n,m\in\N$, $(\lambda_i,z_i)_{i=1}^n\in(\C\times\D)^n$ and $(\mu_j,w_j)_{j=1}^m\in(\C\times\D)^m$. We say that $(\lambda_i,z_i)_{i=1}^n$ is Bloch subordinated to $(\mu_j,w_j)_{j=1}^m$ or, equivalently, Bloch dominated by $(\mu_j,w_j)_{j=1}^m$, and we write $(\lambda_i,z_i)_{i=1}^n\prec (\mu_j,w_j)_{j=1}^m$, whenever 
$$
\sum_{i=1}^n\left|\lambda_i\right|^2\left|g'(z_i)\right|^2\leq \sum_{j=1}^m\left|\mu_j\right|^2\left|g'(w_j)\right|^2
$$
for all $g\in\widehat{\B}(\D)$. 
%Define 
%$$
%(\lambda_i,z_i)_{i=1}^n\prec (\mu_j,w_j)_{j=1}^m \quad\Leftrightarrow \quad\sum_{i=1}^n\left|\lambda_i\right|^2\left|g'(z_i)\right|^2\leq \sum_{j=1}^m\left|\mu_j\right|^2\left|g'(w_j)\right|^2,\quad\forall g\in\widehat{\B}(\D).
%$$
\end{definition}

The following characterization of this criterion will be useful to us later.

\begin{theorem}\label{prop-1}
%Let $n,m\in\N$, $\lambda_1,\ldots,\lambda_n,\mu_1,\ldots,\mu_m\in\C$ and $z_1,\ldots,z_n,w_1,\ldots,w_m\in\D$. 
%Let $n,m\in\N$, $(\lambda_i)_{i=1}^n\in\C^n$, $(\mu_j)_{j=1}^m\in\C^m$, $(z_i)_{i=1}^n\in\D^n$ and $(w_j)_{j=1}^m\in\D^m$. 
Let $n,m\in\N$, $(\lambda_i,z_i)_{i=1}^n\in(\C\times\D)^n$ and $(\mu_j,w_j)_{j=1}^m\in(\C\times\D)^m$. The following assertions are equivalent:
\begin{enumerate}
	\item $(\lambda_i,z_i)_{i=1}^n$ is Bloch dominated by $(\mu_j,w_j)_{j=1}^m$.
	\item There exists an operator $T\in\L(\ell_2^m,\ell_2^n)$ with $\left\|T\right\|\leq 1$, represented by a $n\times m$ complex matrix $(a_{ij})$, such that $\lambda_i\gamma_{z_i}=\sum_{j=1}^m a_{ij}\mu_j\gamma_{w_j}$ for every $1\leq i\leq n$.
\end{enumerate}

%Then $(\lambda_i,z_i)_{i=1}^n\prec (\mu_j,w_j)_{j=1}^m$ if and only if there exists a $n\times m$ complex matrix $(a_{ij})$ such that 
%$$
%\sum_{i=1}^n\left|\sum_{j=1}^m a_{ij}t_j\right|^2\leq\sum_{j=1}^m \left|t_j\right|^2
%$$
%for all $(t_j)_{j=1}^m\in\C^m$, and 
%$$
%\lambda_i\gamma_{z_i}=\sum_{j=1}^m a_{ij}\mu_j\gamma_{w_j}
%$$
%for each $1\leq i\leq n$.
\end{theorem}

\begin{proof}
$(i)\Rightarrow(ii)$: We set
$$
S=\left\{(\mu_jg'(w_j))_{j=1}^m\colon g\in\widehat{\B}(\D)\right\}\subseteq\ell_2^m.
$$
Assume $(\lambda_i,z_i)_{i=1}^n\prec (\mu_j,w_j)_{j=1}^m$. Then the mapping $T_0\colon S\to\ell_2^n$ given by 
$$
T_0\left((\mu_jg'(w_j))_{j=1}^m\right)=(\lambda_ig'(z_i))_{i=1}^n,
$$
is well-defined. Clearly, $T_0$ is linear and since
\begin{align*}
\left\|T_0\left((\mu_jg'(w_j))_{j=1}^m\right)\right\|_{\ell_2^n}
&=\left\|(\lambda_ig'(z_i))_{i=1}^n\right\|_{\ell_2^n}
=\left(\sum_{i=1}^n\left|\lambda_i\right|^2\left|g'(z_i)\right|^2\right)^{\frac{1}{2}}\\
&\leq\left(\sum_{j=1}^m\left|\mu_j\right|^2\left|g'(w_j)\right|^2\right)^{\frac{1}{2}}
=\left\|(\mu_jg'(w_j))_{j=1}^m\right\|_{\ell_2^m}
\end{align*}
for all $(\mu_jg'(w_j))_{j=1}^m\in S$, then $T_0$ is continuous with $\left\|T_0\right\|\leq 1$. By Hahn--Banach Theorem, there exists $T\in\L(\ell_2^m,\ell_2^n)$ with $\left\|T\right\|=\left\|T_0\right\|$ such that $\left.T\right|_S=T_0$. Let $(a_{ij})$ be the representation matrix of $T$ with respect to the canonical bases of $\ell_2^m$ and $\ell_2^n$, that is, 
$$
T((t_j)_{j=1}^m)=\left(\sum_{j=1}^m a_{1j}t_j,\ldots,\sum_{j=1}^m a_{nj}t_j\right)^t
$$
for all $(t_j)_{j=1}^m\in\ell_2^m$. Clearly, 
%we have 
%$$
%\sum_{i=1}^n\left|\sum_{j=1}^m a_{ij}t_j\right|^2=\left\|\widetilde{T}((t_j)_{j=1}^m)\right\|^2_{\ell_2^n}\leq\left\|(t_j)_{j=1}^m\right\|^2_{\ell_2^m}=\sum_{j=1}^m \left|t_j\right|^2
%$$
%for all $(t_j)_{j=1}^m\in\ell_2^m$. Moreover, 
for each $1\leq i\leq n$, we have 
$$
\lambda_i\gamma_{z_i}(g)=\lambda_ig'(z_i)=\sum_{j=1}^m a_{ij}\mu_jg'(w_j)=\sum_{j=1}^m a_{ij}\mu_j\gamma_{w_j}(g)
$$
for all $g\in\widehat{\B}(\D)$, and thus $\lambda_i\gamma_{z_i}=\sum_{j=1}^m a_{ij}\mu_j\gamma_{w_j}$, as required. 

$(ii)\Rightarrow(i)$: Given $g\in\widehat{\B}(\D)$, we have $(\lambda_i,g'(z_i))_{i=1}^n=T((\mu_jg'(w_j))_{j=1}^m)$, hence
\begin{align*}
\sum_{i=1}^n\left|\lambda_i\right|^2\left|g'(z_i)\right|^2
&=\left\|(\lambda_i,g'(z_i))_{i=1}^n\right\|^2_{\ell_2^n}
\leq \left\|T\right\|^2\left\|(\mu_jg'(w_j))_{j=1}^m\right\|^2_{\ell_2^m}\\
&\leq \left\|(\mu_jg'(w_j))_{j=1}^m\right\|^2_{\ell_2^m}
=\sum_{j=1}^m\left|\mu_j\right|^2\left|g'(w_j)\right|^2.
\end{align*}
%We can reverse the preceding argument to prove the another implication.
\end{proof}

The following result is a Bloch version of a classical Kwapie\'n's theorem \cite{Kwa-72} by means of Bloch dominated finite sequences of elements of $\C\times\D$. 

As usual, $B_X$ denotes the closed unit ball of a Banach space $X$. Given a Hausdorff compact topological space $T$, $\Cu(T,\mathbb{K})$ stands for the Banach space of all scalar-valued continuous functions on $T$, endowed with the supremum norm.  

\begin{theorem}\label{theo-8}
Let $X$ be a complex Banach space and $f\in\H(\D,X)$. The following statements are equivalent:
\begin{enumerate}
	\item $f$ factors through a Hilbert space.
	\item There exists a constant $c\geq 0$ such that 
$$
\sum_{i=1}^n\left|\lambda_i\right|^2\left\|f'(z_i)\right\|^2\leq c^2\sum_{j=1}^m\frac{\left|\mu_j\right|^2}{(1-|w_j|^2)^2},%\sum_{j=1}^m \left|\mu_j\right|^2\left\|f'(w_j)\right\|^2
$$
whenever $(\lambda_i,z_i)_{i=1}^n\prec (\mu_j,w_j)_{j=1}^m$ with $n,m\in\N$, $(\lambda_i,z_i)_{i=1}^n\in(\C\times\D)^n$ and $(\mu_j,w_j)_{j=1}^m\in(\C\times\D)^m$. 
%with $n,m\in\N$, $(\lambda_i)_{i=1}^n\in\C^n$, $(\mu_j)_{j=1}^m\in\C^m$, $(z_i)_{i=1}^n\in\D^n$ and $(w_j)_{j=1}^m\in\D^m$. 
%with $n,m\in\N$, $\lambda_1,\ldots,\lambda_n,\mu_1,\ldots,\mu_m\in\C$ and $z_1,\ldots,z_n,w_1,\ldots,w_m\in\D$. 
\end{enumerate}
In this case, $\gamma^{\B}_2(f)$ is the minimum of all constants $c$ satisfying the preceding inequality. 
\end{theorem} 

\begin{proof}
$(i)\Rightarrow(ii)$: Assume $f=T\circ g$, where $g\in\widehat{\B}(\D,H)$ and $T\in\L(H,X)$ for some Hilbert space $H$ with inner product $\innerproduct{\cdot}{\cdot}$.  
%Let $n,m\in\N$, $(\lambda_i)_{i=1}^n\in\C^n$, $(\mu_j)_{j=1}^m\in\C^m$, $(z_i)_{i=1}^n\in\D^n$ and $(w_j)_{j=1}^m\in\D^m$ 
Let $n,m\in\N$, $(\lambda_i,z_i)_{i=1}^n\in(\C\times\D)^n$ and $(\mu_j,w_j)_{j=1}^m\in(\C\times\D)^m$ such that $(\lambda_i,z_i)_{i=1}^n\prec (\mu_j,w_j)_{j=1}^m$. We claim that for each $v\in H$, the function $h_v\colon z\mapsto \innerproduct{g(z)}{v}$ from $\D$ into $\C$ belongs to $\widehat{\B}(\D)$. Indeed, $h_v(0)=\innerproduct{g(0)}{v}=\innerproduct{0}{v}=0$ and $h_v$ is Bloch with $\rho_\B(h_v)\leq \rho_\B(g)\left\|v\right\|$ since
$$
(1-|z|^2)\left|h_v'(z)\right|=(1-|z|^2)\left|\innerproduct{g'(z)}{v}\right|\leq (1-|z|^2)\left\|g'(z)\right\|\left\|v\right\|\leq \rho_\B(g)\left\|v\right\|
$$
for all $z\in\D$. Consequently, we have 
$$
\sum_{i=1}^n\left|\lambda_i\right|^2\left|\innerproduct{g'(z_i)}{v}\right|^2=\sum_{i=1}^n\left|\lambda_i\right|^2\left|h_v'(z_i)\right|^2
\leq \sum_{j=1}^m \left|\mu_j\right|^2\left|h_v'(w_j)\right|^2=\sum_{j=1}^m\left|\mu_j\right|^2\left|\innerproduct{g'(w_j)}{v}\right|^2
$$
for all $v\in H$. Let $(v_\alpha)_{\alpha\in A}$ be an orthonormal basis of $H$. Since 
$$
\left\|v\right\|^2=\sum_{\alpha\in A}\left|\innerproduct{v}{v_\alpha}\right|^2
$$
for all $v\in H$, we deduce  
$$
\sum_{i=1}^n\left|\lambda_i\right|^2\left\|g'(z_i)\right\|^2\leq\sum_{j=1}^m \left|\mu_j\right|^2\left\|g'(w_j)\right\|^2.
$$
%\begin{align*}
%\sum_{i=1}^n\left|\lambda_i\right|^2\left\|g'(z_i)\right\|^2
%&=\sum_{i=1}^n\left|\lambda_i\right|^2\sum_{\alpha\in A}\left|<g'(z_i),v_\alpha\right|^2
%=\sum_{\alpha\in A}\sum_{i=1}^n\left|\lambda_i\right|^2\left|<g'(z_i),v_\alpha\right|^2\\
%&\leq \sum_{\alpha\in A}\sum_{j=1}^m \left|\mu_j\right|^2\left|<g'(w_j),v_\alpha>\right|^2
%=\sum_{j=1}^m \left|\mu_j\right|^2\sum_{\alpha\in A}\left|<g'(w_j),v_\alpha>\right|^2
%=\sum_{j=1}^m \left|\mu_j\right|^2\left\|g'(w_j)\right\|^2.
%
%\end{align*}
Therefore we have 
\begin{align*}
%\sum_{i=1}^n\frac{\left|\lambda_i\right|^2}{(1-\left\|f'(z_i)\right\|^2)^2}&\leq
\sum_{i=1}^n\left|\lambda_i\right|^2\left\|f'(z_i)\right\|^2&=\sum_{i=1}^n\left|\lambda_i\right|^2\left\|T(g'(z_i))\right\|^2\\
&\leq\left\|T\right\|^2\sum_{i=1}^n\left|\lambda_i\right|^2\left\|g'(z_i)\right\|^2\leq\left\|T\right\|^2\sum_{j=1}^m\left|\mu_j\right|^2\left\|g'(w_j)\right\|^2\\
%=\left\|T\right\|^2\sum_{j=1}^m\frac{\left|\mu_j\right|^2}{(1-|w_j|^2)^2}(1-|w_j|^2)^2\left\|g'(w_j)\right\|^2
&\leq\left\|T\right\|^2\rho_{\B}(g)^2\sum_{j=1}^m\frac{\left|\mu_j\right|^2}{(1-|w_j|^2)^2}.
\end{align*}
And taking the infimum over all factorizations of $f'$, we deduce  
$$
\sum_{i=1}^n\left|\lambda_i\right|^2\left\|f'(z_i)\right\|^2\leq\gamma^{\B}_2(f)\sum_{j=1}^m\frac{\left|\mu_j\right|^2}{(1-|w_j|^2)^2}.
$$

$(ii)\Rightarrow(i)$: Assume that $(ii)$ holds. From Theorem \ref{main-theo}, we can consider the Hausdorff compact space $(B_{\widehat{\B}(\D)},w^*)$. For each $z\in\D$, define $\eta_z\colon B_{\widehat{\B}(\D)}\to\C$ by $\eta_z(k)=k'(z)$ for all $k\in B_{\widehat{\B}(\D)}$. Clearly, $\eta_z\in\Cu(B_{\widehat{\B}(\D)},\C)$ with $\left\|\eta_z\right\|_\infty=1/(1-|z|^2)$. Consider the set $K_1\subseteq\Cu(B_{\widehat{\B}(\D)},\R)$ formed by all functions of the form:
$$
\Phi=\sum_{i=1}^n\left|\lambda_i\right|^2\left|\eta_{z_i}\right|^2-\sum_{j=1}^m\left|\mu_j\right|^2\left|\eta_{w_j}\right|^2,
$$
where $(\lambda_i,z_i)_{i=1}^n\in(\C\times\D)^n$ and $(\mu_j,w_j)_{j=1}^m\in(\C\times\D)^m$ satisfy  
$$
\sum_{i=1}^n\left|\lambda_i\right|^2\left\|f'(z_i)\right\|^2>c^2\sum_{j=1}^m\frac{\left|\mu_j\right|^2}{(1-|w_j|^2)^2}
$$
By concatenation, it is easy to see that $K_1$ is convex. Condition in the statement (ii) assures that $\sup\left\{\Phi(k)\colon k\in B_{\widehat{\B}(\D)}\right\}>0$ for each $\Phi\in K_1$. Hence $K_1\cap K_2=\emptyset$, where 
$$
K_2=\left\{\Psi\in\Cu(B_{\widehat{\B}(\D)},\R) \colon \sup\left\{\Psi(k)\colon k\in B_{\widehat{\B}(\D)}\right\}<0\right\}.
$$
Clearly, $K_2$ is an open convex subset of $\Cu(B_{\widehat{\B}(\D)},\R)$. Therefore, the Hahn--Banach Separation Theorem and the Riesz Representation Theorem provide a non-zero signed Borel measure $\nu$ on $B_{\widehat{\B}(\D)}$ and a real number $\alpha$ such that
$$
\int_{B_{\widehat{\B}(\D)}}\Psi(k)\ d\nu(k)<\alpha\leq\int_{B_{\widehat{\B}(\D)}}\Phi(k)\ d\nu(k)\qquad (\Psi\in K_2,\; \Phi\in K_1).
$$
%for all $\Psi\in K_2$ and $\Phi\in K_1$. 
Since $r\Psi\in K_2$ for all $r>0$, it follows that $\alpha\geq 0$. By the definition of $K_2$, we obtain  
$$
\int_{B_{\widehat{\B}(\D)}}\Psi(k)\ d\nu(k)\leq 0\qquad (\Psi\in K_2).
$$
In particular,  
$$
\nu(B_{\widehat{\B}(\D)})=\int_{B_{\widehat{\B}(\D)}}\ d\nu(k)\geq 0,
$$
(in fact, $\nu(B_{\widehat{\B}(\D)})>0$ since $\nu\neq 0$), and therefore $\nu$ is a positive measure. Moreover, we have
$$
\int_{B_{\widehat{\B}(\D)}}\Phi(k)\ d\nu(k)\geq 0\qquad (\Phi\in K_1).
$$
%for all $\Phi\in K_1$. 

Now, note that if $(\mu,w)\in\C\times\D$ and $\left|\mu\right|/(1-|w|^2)\leq 1$, we have  
$$
\int_{B_{\widehat{\B}(\D)}}\left|\mu\right|^2\left|k'(w)\right|^2\ d\nu(k)\leq\nu(B_{\widehat{\B}(\D)}),
$$
and therefore there exists 
$$
a:=\sup\left\{\left(\int_{B_{\widehat{\B}(\D)}}\left|\mu\right|^2\left|k'(w)\right|^2\ d\nu(k)\right)^{\frac{1}{2}}\colon (\mu,w)\in \C\times\D,\, \frac{\left|\mu\right|}{1-|w|^2}\leq 1\right\}.
$$
Furthermore, $a>0$. Otherwise, for all $(\mu,w)\in\C\times\D$ such that $\left|\mu\right|/(1-|w|^2)\leq 1$, we would have  
$$
\left\|\mu\eta_w\right\|_{L_2(\nu)}=\left(\int_{B_{\widehat{\B}(\D)}}\left|\mu\right|^2\left|k'(w)\right|^2\ d\nu(k)\right)^{\frac{1}{2}}=0,
$$
that is, $\mu\eta_w=0$. In particular, $(1-|w|^2)\eta_w=0$ or, equivalently, $\eta_w=0$, but $\eta_w(f_w)=f'_w(w)=1/(1-|w|^2)\neq 0$, a contradiction. Taking $\widehat{\nu}:=(c^2/a^2)\nu$ instead of $\nu$, we obtain   
$$
\sup\left\{\left(\int_{B_{\widehat{\B}(\D)}}\left|\mu\right|^2\left|k'(w)\right|^2\ d\widehat{\nu}(k)\right)^{\frac{1}{2}}\colon (\mu,w)\in \C\times\D,\, \frac{\left|\mu\right|}{1-|w|^2}\leq 1\right\}=c.
$$

Define $h\colon\D\to L_2(\widehat{\nu})$ by $h(z)=j_2(\eta_z)$ for all $z\in\D$, where $j_2\colon\Cu(B_{\widehat{\B}(\D)},\C)\to L_2(\widehat{\nu})$ is the formal inclusion operator. Note that $h\in\H(\D,L_2(\widehat{\nu}))$. %with $h'(z)(k)=k''(z)$ for all $z\in\D$ and $k\in L_2(\nu)$.  
By \cite[Lemma 2.9]{JimRui-22}, there exists a mapping $g\in\H(\D,L_2(\widehat{\nu}))$ with $g(0)=0$ such that $g'=h$. In fact, $g\in\widehat{\B}(\D,L_2(\widehat{\nu}))$ with $\rho_\B(g)\leq c$ since 
%$$
%\left\|h(z)\right\|_{L_2(\widehat{\nu})}
%=\left(\int_{B_{\widehat{\B}(\D)}}\left|\eta_z(k)\right|^2\  d\widehat{\nu}(k)\right)^{\frac{1}{2}}
%=\left(\int_{B_{\widehat{\B}(\D)}}\left|k'(z)\right|^2\  d\widehat{\nu}(k)\right)^{\frac{1}{2}}
%\leq \frac{c}{1-|z|^2}
%$$
%and thus 
$$
(1-|z|^2)\left\|g'(z)\right\|_{L_2(\widehat{\nu})}=(1-|z|^2)\left\|h(z)\right\|_{L_2(\widehat{\nu})}=(1-|z|^2)\left(\int_{B_{\widehat{\B}(\D)}}\left|k'(z)\right|^2\  d\widehat{\nu}(k)\right)^{\frac{1}{2}}\leq c
$$
for all $z\in\D$. Consider now the Hilbert space $H:=\overline{\lin}(h(\D))\subseteq L_2(\widehat{\nu})$ and the linear operator $T\colon H\to X$ defined by $T(j_2(\eta_z))=f'(z)$ for all $z\in\D$. We will now show that $T$ is continuous with $\left\|T\right\|\leq 1$. First, we claim that 
$$
(\lambda,z)\in \C\times\D,\; \left|\lambda\right|^2\left\|f'(z)\right\|^2>c^2\quad \Rightarrow\quad \int_{B_{\widehat{\B}(\D)}}\left|\lambda\right|^2\left|k'(z)\right|^2\ d\widehat{\nu}(k)\geq c^2.
$$
To see this, if $(\mu,w)\in \C\times\D$ satisfies that $\left|\mu\right|/(1-|w|^2)\leq 1$, then we have  
$$
\left|\lambda\right|^2\left\|f'(z)\right\|^2>c^2\frac{\left|\mu\right|^2}{(1-|w|^2)^2},
$$ 
hence the property of integration of the functions $\Phi\in K_1$ yields 
$$
\int_{B_{\widehat{\B}(\D)}}\left|\lambda\right|^2\left|k'(z)\right|^2\ d\widehat{\nu}(k)\geq \int_{B_{\widehat{\B}(\D)}}\left|\mu\right|^2\left|k'(w)\right|^2\ d\widehat{\nu}(k),
$$
and taking the supremum over all such $(\mu,w)\in \C\times\D$, we conclude that 
$$
\int_{B_{\widehat{\B}(\D)}}\left|\lambda\right|^2\left|k'(z)\right|^2\ d\widehat{\nu}(k)\geq c^2.
$$
Our claim implies that 
$$
\left|\lambda\right|^2\left\|f'(z)\right\|^2\leq \int_{B_{\widehat{\B}(\D)}}\left|\lambda\right|^2\left|k'(z)\right|^2\ d\widehat{\nu}(k). 
$$
Indeed, if $(\lambda,z)\in\C\times\D$, $b:=\left|\lambda\right|^2\left\|f'(z)\right\|^2$ and $\varepsilon>0$, then 
$$
\left|\lambda\right|^2\frac{c^2+\varepsilon}{b}\left\|f'(z)\right\|^2=c^2+\varepsilon>c^2.
$$
Our claim yields
$$
\int_{B_{\widehat{\B}(\D)}}\left|\lambda\right|^2\frac{c^2+\varepsilon}{b}\left|k'(z)\right|^2\ d\widehat{\nu}(k)\geq c^2,
$$ 
that is, 
$$
\int_{B_{\widehat{\B}(\D)}}\left|\lambda\right|^2\left|k'(z)\right|^2\ d\widehat{\nu}(k)\geq \frac{c^2}{c^2+\varepsilon}b,
$$
and since $\varepsilon$ is arbitrary, we conclude that 
$$
\int_{B_{\widehat{\B}(\D)}}\left|\lambda\right|^2\left|k'(z)\right|^2\ d\widehat{\nu}(k)\geq b.
$$
%Therefore% obtained inequality can be rewritten as  
%$$
%\left\|f'(z)\right\|\leq \left(\int_{B_{\widehat{\B}(\D)}}\left|k'(z)\right|^2\ d\widehat{\nu}(k)\right)^{\frac{1}{2}}
%=\left(\int_{B_{\widehat{\B}(\D)}}\left|\eta_z(k)\right|^2\ d\widehat{\nu}(k)\right)^{\frac{1}{2}}=\left\|j_2(\eta_z)\right\|_{L_2(\widehat{\nu})},
%$$
%and 
We can adjust $\widehat{\nu}$ so that $\widehat{\nu}(B_{\widehat{\B}(\D)})=1$. Given $n\in\mathbb{N}$, $\alpha_1,\ldots,\alpha_n\in\mathbb{C}^*$ and $z_1,\ldots,z_n\in\D$, we have  
\begin{align*}
\left\|T\left(\sum_{i=1}^n\alpha_ij_2(\eta_{z_i})\right)\right\|&=\left\|\sum_{i=1}^n\alpha_iT(j_2(\eta_{z_i}))\right\|=\left\|\sum_{i=1}^n\alpha_if'(z_i)\right\|\\
&\leq\sum_{i=1}^n\left|\alpha_i\right|\left\|f'(z_i)\right\|\leq \sum_{i=1}^n\left|\alpha_i\right|\left(\int_{B_{\widehat{\B}(\D)}}\left|k'(z)\right|^2\ d\widehat{\nu}(k)\right)^{\frac{1}{2}}\\
%\sum_{i=1}^n\left|\alpha_i\right|\left\|j_2(\eta_{z_i})\right\|_{L_2(\widehat{\nu})}\\
&\leq\sum_{i=1}^n\frac{\left|\alpha_i\right|}{1-|z_i|^2}=\left|\sum_{i=1}^n\alpha_i\left(\frac{\overline{\alpha_i}}{\left|\alpha_i\right|}f_{z_i}\right)'(z_i)\right|\\
&=\sup_{k\in B_{\widehat{\B}(\D)}}\left|\sum_{i=1}^n\alpha_i k'(z_i)\right|=\sup_{k\in B_{\widehat{\B}(\D)}}\left|\sum_{i=1}^n\alpha_i \eta_{z_i}(k)\right|=\left\|\sum_{i=1}^n\alpha_i\eta_{z_i}\right\|_\infty\\
&=\left\|j_2\left(\sum_{i=1}^n\alpha_i\eta_{z_i}\right)\right\|_{L_2(\widehat{\nu})}=\left\|\sum_{i=1}^n\alpha_ij_2(\eta_{z_i})\right\|_{L_2(\widehat{\nu})},
\end{align*}
and thus $\left\|T\right\|\leq 1$. Clearly, $f'=T\circ g'$. Hence $f\in\Gamma^{\widehat{\B}}_2(\D,X)$ with $\gamma^{\B}_2(f)\leq \left\|T\right\|\rho_\B(g)\leq c$.
\end{proof}

Theorem \ref{theo-8} admits the following remarkable reformulation.

\begin{corollary}\label{cor-8}
Let $X$ be a complex Banach space and $f\in\H(\D,X)$. The following assertions are equivalent:
\begin{enumerate}
\item $f$ factors through a Hilbert space.
\item There exists a constant $c\geq 0$ such that for every $n\in\N$, every $n\times n$ unitary complex matrix $(a_{ij})$ and every $(\mu_j,w_j)_{j=1}^n\in(\C\times\D)^n$, we have 
$$
\sum_{i=1}^n\left\|\sum_{j=1}^n\mu_j a_{ij}f'(w_j)\right\|^2\leq c^2\sum_{j=1}^n\frac{\left|\mu_j\right|^2}{(1-|w_j|^2)^2}.
$$ 
\end{enumerate}
In this case, $\gamma^{\B}_2(f)$ is the minimum of all constants $c$ satisfying the preceding inequality. 
\end{corollary} 

\begin{proof}
$(i)\Rightarrow(ii)$: If $(i)$ holds, then Theorem \ref{theo-8} provides a constant $c\geq 0$ such that 
$$
\sum_{j=1}^n\left|\mu_j\right|^2\left\|f'(w_j)\right\|^2\leq c^2\sum_{j=1}^n\frac{\left|\mu_j\right|^2}{(1-|w_j|^2)^2}
$$
for all $n\in\N$ and $(\mu_j,w_j)_{j=1}^n\in(\C\times\D)^n$. Let $(a_{ij})$ be a $n\times n$ unitary complex matrix. Using the Cauchy--Schwarz Inequality, we have
\begin{align*}
\sum_{i=1}^n\left\|\sum_{j=1}^n\mu_j a_{ij}f'(w_j)\right\|^2&\leq \sum_{i=1}^n\left(\sum_{j=1}^n\left|\mu_j\right|\left|a_{ij}\right|\left\|f'(w_j)\right\|\right)^2\\
&\leq \sum_{i=1}^n\left(\sum_{j=1}^n\left|\mu_j\right|^2\left\|f'(w_j)\right\|^2\right)\left(\sum_{j=1}^n\left|a_{ij}\right|^2\right)\\
&=\left(\sum_{j=1}^n\left|\mu_j\right|^2\left\|f'(w_j)\right\|^2\right)\left(\sum_{i=1}^n\sum_{j=1}^n\left|a_{ij}\right|^2\right)\\
&\leq \sum_{j=1}^n\left|\mu_j\right|^2\left\|f'(w_j)\right\|^2\leq c^2\sum_{j=1}^n\frac{\left|\mu_j\right|^2}{(1-|w_j|^2)^2}.
\end{align*}

$(ii)\Rightarrow(i)$: Assume there is a constant $c\geq 0$ such that  
$$
\sum_{i=1}^n\left\|\sum_{j=1}^n\mu_j b_{ij}f'(w_j)\right\|^2\leq c^2\sum_{j=1}^n\frac{\left|\mu_j\right|^2}{(1-|w_j|^2)^2}
$$
for any $n\in\N$, any $n\times n$ unitary complex matrix $(b_{ij})$ and any $(\mu_j,w_j)_{j=1}^n\in(\C\times\D)^n$. It suffices if we show that the assertion $(ii)$ of Theorem \ref{theo-8} holds. For it, let $n,m\in\N$, $(\lambda_i,z_i)_{i=1}^n\in(\C\times\D)^n$ and $(\mu_j,w_j)_{j=1}^m\in(\C\times\D)^m$ such that $(\lambda_i,z_i)_{i=1}^n\prec (\mu_j,w_j)_{j=1}^m$. We can suppose $m=n$ (otherwise we add some zeros). By Theorem \ref{prop-1}, there is a $n\times n$ complex matrix $A=(a_{ij})$ in $B_{\L(\ell_2^n)}$ such that   
%$$
%\sum_{i=1}^n\left|\sum_{j=1}^n a_{ij}t_j\right|^2\leq\sum_{j=1}^n \left|t_j\right|^2
%$$
%for all $(t_j)_{j=1}^n\in\C^n$, and 
$$
\lambda_i\gamma_{z_i}=\sum_{j=1}^n a_{ij}\mu_j\gamma_{w_j}
$$
for every $1\leq i\leq n$. Then, if $\phi\colon B_{\L(\ell_2^n)}\to\R^+$ is the convex functional defined by 
$$
\phi(B)=\sum_{i=1}^n\left\|\sum_{j=1}^n b_{ij}\mu_jf'(w_j)\right\|^2\qquad (B=(b_{ij})\in B_{\L(\ell_2^n)}),
$$
we have 
$$
\phi(A)=\sum_{i=1}^n\left\|\sum_{j=1}^n a_{ij}\mu_j\gamma_{w_j}(f)\right\|^2=\sum_{i=1}^n\left\|\lambda_i\gamma_{z_i}(f)\right\|^2=\sum_{i=1}^n\left|\lambda_i\right|^2\left\|f'(z_i)\right\|^2.
$$
Moreover, by hypothesis, we have  
$$
\phi(B)\leq c^2\sum_{j=1}^m\frac{\left|\mu_j\right|^2}{(1-|w_j|^2)^2}
$$
for any $n\times n$ unitary complex matrix $B=(b_{ij})$. Since the extreme points of $B_{\L(\ell_2^n)}$ are precisely the unitary complex matrices, then $B_{\L(\ell_2^n)}$ coincides with the convex hull of such matrices by the Minkowski--Carath\'eodory Theorem. Therefore, we deduce that  
$$
\phi(A)\leq c^2\sum_{j=1}^m\frac{\left|\mu_j\right|^2}{(1-|w_j|^2)^2},
$$
and this completes the proof.
\end{proof}

%%%%%%%%%%%%%%%%%%%%%%%%%%%%%%%%%%%%%%%%%%%%%%%%%%%%%%%%%%%%%%%%%%%%%%%%%%%%%%%%%%%%%%%%%%%%%%%%%%%%%%%%%%%%%%%%%%%%%%%%%%%%%%%%%%%%%%%%%%%%%%%%%%%%%%%%%%%%%%%%%%%%%%%%%%%%%%%%%%

\section{Duality}\label{5}

Our aim in this last section is to show that the space $(\Gamma^{\widehat{\B}}_2(\D,X^*),\gamma^{\B}_2)$ can be canonically identified with the dual of the completion of the space $\lin(\Gamma(\D))\otimes X$ with respect to a convenient norm.
%NO es correcto poner: $\G(\D)\otimes X$

We first recall some concepts introduced in \cite{CabJimRui-23}. In what follows, given a complex Banach space $X$, we will sometimes write $\langle x^*,x\rangle=x^*(x)$ when $x\in X$ and $x^*\in X^*$. Given $z\in\D$ and $x\in X$, the functional $\gamma_z\otimes x\colon\widehat{\B}(\D,X^*)\to\mathbb{C}$ defined by  
$$
\left(\gamma_z\otimes x\right)(f)=\left\langle f'(z),x\right\rangle\qquad \left(f\in\widehat{\B}(\D,X^*)\right), 
$$
is linear and continuous with $\left\|\gamma_z\otimes x\right\|=\left\|x\right\|/(1-|z|^2)$. We set the linear space 
$$
\lin(\Gamma(\D))\otimes X=\lin\left\{\gamma_z\otimes x\colon z\in\D,\, x\in X\right\}\subseteq\widehat{\B}(\D,X^*)^*,
$$
whose elements are called $X$-valued Bloch molecules on $\D$.

Our study of the duality of $(\Gamma^{\widehat{\B}}_2(\D,X^*),\gamma^{\B}_2)$ requires to consider the following norm on these molecules. 

\begin{definition}
Let $X$ be a complex Banach space. Given $\gamma\in\lin(\Gamma(\D))\otimes X$, define 
$$
w_2^{\B}(\gamma)=\inf\left\{\left(\sum_{i=1}^n\left\|x_i\right\|^2\right)^{\frac{1}{2}}\left(\sum_{j=1}^m\frac{\left|\mu_j\right|^2}{(1-|w_j|^2)^2}\right)^{\frac{1}{2}}\right\},
$$
where the infimum is taken over all such representations of $\gamma$ in the form $\sum_{i=1}^n\lambda_i\gamma_{z_i}\otimes x_i$ such that $(\lambda_i,z_i)_{i=1}^n\prec (\mu_j,w_j)_{j=1}^m$ with  %$n,m\in\N$, $(\lambda_i)_{i=1}^n\in\C^n$, $(\mu_j)_{j=1}^m\in\C^m$, $(z_i)_{i=1}^n\in\D^n$, $(w_j)_{j=1}^m\in\D^m$ and $(x_i)_{i=1}^n\in X^n$. 
$n,m\in\N$, $(\lambda_i,z_i)_{i=1}^n\in(\C\times\D)^n$, $(\mu_j,w_j)_{j=1}^m\in(\C\times\D)^m$ and $(x_i)_{i=1}^n\in X^n$.
\end{definition}

The concept of Bloch reasonable cross-norm was introduced in \cite{CabJimRui-23}, inspired by the analogue in the theory of tensor product (see, for example, \cite{Rya-02}).

%Let us recall (see \cite{CabJimRui-23}) that a norm $\alpha$ on $\lin(\Gamma(\D))\otimes X$ is a Bloch reasonable cross-norm if:
%\begin{enumerate}
%\item $\alpha(\gamma_z\otimes x)\leq \left\|\gamma_z\right\|\left\|x\right\|$ for all $z\in\D$ and $x\in X$,
%\item For every $g\in\widehat{\B}(\D)$ and $x^*\in X^*$, the linear functional $g\otimes x^*\colon\lin(\Gamma(\D))\otimes_\alpha X\to \C$, defined by $(g\otimes x^*)(\gamma_z\otimes x)=g'(z)x^*(x)$, is bounded with $\left\|g\otimes x^*\right\|\leq \rho_\B(g)\left\|x^*\right\|$.
%\end{enumerate}

\begin{theorem}\label{teo-che-norms}
$w_2^\B$ is a Bloch reasonable cross-norm on $\lin(\Gamma(\D))\otimes X$ for any complex Banach space $X$.
\end{theorem}

\begin{proof}
Let $\gamma\in\lin(\Gamma(\D))\otimes X$ and let $\sum_{i=1}^n\lambda_i\gamma_{z_i}\otimes x_i$ be a representation of $\gamma$ with $n\in\N$, $(\lambda_i)_{i=1}^n\in\C^n$, $(z_i)_{i=1}^n\in\D^n$ and $(x_i)_{i=1}^n\in X^n$. Let $(\mu_j,w_j)_{j=1}^m\succ (\lambda_i,z_i)_{i=1}^n$ with $m\in\N$ and $(\mu_j,w_j)_{j=1}^m\in(\C\times\D)^m$. 

Clearly, $w_2^{\B}(\gamma)\geq 0$. Given $\lambda\in\C$, since $\lambda\gamma=\sum_{i=1}^n (\lambda\lambda_i)\gamma_{z}\otimes x_i$ and $(\lambda\lambda_i,z_i)_{i=1}^n\prec (\lambda\mu_j,w_j)_{j=1}^m$, we have 
$$
w_2^\B(\lambda\gamma)\leq\left(\sum_{i=1}^n\left\|x_i\right\|^2\right)^{\frac{1}{2}}\left(\sum_{j=1}^m\frac{\left|\lambda\mu_j\right|^2}{(1-|w_j|^2)^2}\right)^{\frac{1}{2}}
=\left|\lambda\right|\left(\sum_{i=1}^n\left\|x_i\right\|^2\right)^{\frac{1}{2}}\left(\sum_{j=1}^m\frac{\left|\mu_j\right|^2}{(1-|w_j|^2)^2}\right)^{\frac{1}{2}}.
$$
If $\lambda=0$, we obtain $w_2^\B(\lambda\gamma)=0=\left|\lambda\right|w_2^{\B}(\gamma)$. For $\lambda\neq 0$, since the preceding inequality holds for every representation of $\gamma$, we deduce that $w_2^\B(\lambda\gamma)\leq\left|\lambda\right|w_2^{\B}(\gamma)$. For the converse inequality, note that $w_2^{\B}(\gamma)=w_2^\B(\lambda^{-1}(\lambda\gamma))\leq |\lambda^{-1}|w_2^\B(\lambda\gamma)$ by the proved inequality, hence $\left|\lambda\right|w_2^{\B}(\gamma)\leq w_2^\B(\lambda\gamma)$ and thus $w_2^\B(\lambda\gamma)=\left|\lambda\right|w_2^{\B}(\gamma)$.

We now prove the triangular inequality of $w_2^\B$. Let $\gamma_1,\gamma_2\in\lin(\Gamma(\D))\otimes X$ and let $\varepsilon>0$. %If $\gamma_1=0$ or $\gamma_2=0$, there is nothing to prove. Assume $\gamma_1\neq 0\neq \gamma_2$. 
We can choose a representation $\gamma_1=\sum_{i=1}^n\lambda_{i}\gamma_{z_{i}}\otimes x_{i}$ and $(\mu_j,w_j)_{j=1}^m\succ(\lambda_i,z_i)_{i=1}^n$ such that
$$
\left(\sum_{i=1}^n\left\|x_{i}\right\|^2\right)^{\frac{1}{2}}\leq \left(w_2^\B(\gamma_1)+\varepsilon\right)^{\frac{1}{2}}
\qquad \text{and}\qquad 
\left(\sum_{j=1}^m\frac{\left|\mu_j\right|^2}{(1-|w_j|^2)^2}\right)^{\frac{1}{2}}\leq \left(w_2^\B(\gamma_1)+\varepsilon\right)^{\frac{1}{2}}.
$$
Similarly, take a representation $\gamma_2=\sum_{i=n+1}^{n+k}\lambda_{i}\gamma_{z_{i}}\otimes x_{i}$ and $(\mu_j,w_j)_{j=m+1}^{m+l}\succ(\lambda_i,z_i)_{i=n+1}^{n+k}$ such that 
$$
\left(\sum_{i=n+1}^{n+k}\left\|x_{i}\right\|^2\right)^{\frac{1}{2}}\leq \left(w_2^\B(\gamma_2)+\varepsilon\right)^{\frac{1}{2}}
\qquad \text{and}\qquad 
\left(\sum_{j=m+1}^{m+l}\frac{\left|\mu_j\right|^2}{(1-|w_j|^2)^2}\right)^{\frac{1}{2}}\leq \left(w_2^\B(\gamma_2)+\varepsilon\right)^{\frac{1}{2}}.
$$
Then $\gamma_1+\gamma_2=\sum_{i=1}^{n+k}\lambda_{i}\gamma_{z_{i}}\otimes x_{i}$, $(\mu_j,w_j)_{j=1}^{m+l}\succ(\lambda_i,z_i)_{i=1}^{n+k}$ and  
$$
\left(\sum_{i=n+1}^{n+k}\left\|x_{i}\right\|^2\right)^{\frac{1}{2}}\left(\sum_{j=1}^{m+l}\frac{\left|\mu_j\right|^2}{(1-|w_j|^2)^2}\right)^{\frac{1}{2}}
\leq w_2^\B(\gamma_1)+w_2^\B(\gamma_2)+2\varepsilon.
$$
hence $w_2^\B(\gamma_1+\gamma_2)\leq w_2^\B(\gamma_1)+w_2^\B(\gamma_2)+2\varepsilon$, and thus $w_2^\B(\gamma_1+\gamma_2)\leq w_2^\B(\gamma_1)+w_2^\B(\gamma_2)$ by the arbitrariness of $\varepsilon$. 

Hence $w_2^\B$ is a seminorm. To prove that it is a norm, note first that 
\begin{align*}
\left|\sum_{i=1}^n \lambda_i g'(z_i)x^*(x_i)\right|
&\leq\sum_{i=1}^n\left|\lambda_i\right|\left|g'(z_i)\right|\left\|x_i\right\|\\
&\leq\left(\sum_{i=1}^{n}\left\|x_i\right\|^2\right)^{\frac{1}{2}}\left(\sum_{i=1}^{n}\left|\lambda_i\right|^{2}\left|g'(z_i)\right|^{2}\right)^{\frac{1}{2}}\\
&\leq\left(\sum_{i=1}^{n}\left\|x_i\right\|^2\right)^{\frac{1}{2}}\left(\sum_{j=1}^{m}\left|\mu_j\right|^{2}\left|g'(w_j)\right|^{2}\right)^{\frac{1}{2}}\\
&\leq\left(\sum_{i=1}^n\left\|x_i\right\|^2\right)^{\frac{1}{2}}\left(\sum_{j=1}^m\frac{\left|\mu_j\right|^2}{(1-|w_j|^2)^2}\right)^{\frac{1}{2}}
\end{align*}
for any $g\in B_{\widehat{\B}(\D)}$ and $x^*\in B_{X^*}$. 
%, by applying Cauchy--Schwarz Inequality. 
Since the first quantity $\left|\sum_{i=1}^n \lambda_i g'(z_i)x^*(x_i)\right|$ does not depend on the representation of $\gamma$ because 
$$
\sum_{i=1}^n \lambda_i g'(z_i)x^*(x_i)=\left(\sum_{i=1}^n\lambda_i\gamma_{z_i}\otimes x_i\right)(g\cdot x^*)=\gamma(g\cdot x^*),
$$
taking the infimum over all representations of $\gamma$ we deduce that  
$$
\left|\sum_{i=1}^n \lambda_i g'(z_i)x^*(x_i)\right|\leq w_2^{\B}(\gamma)
$$
for any $g\in B_{\widehat{\B}(\D)}$ and $x^*\in B_{X^*}$. Now, if $w_2^{\B}(\gamma)=0$, the preceding inequality yields 
$$
\left(\sum_{i=1}^n \lambda_i x^*(x_i)\gamma_{z_i}\right)(g)=\sum_{i=1}^n \lambda_i x^*(x_i) g'(z_i)=0
$$
for all $g\in B_{\widehat{\B}(\D)}$ and $x^*\in B_{X^*}$. For each $x^*\in B_{X^*}$, this implies that $\sum_{i=1}^n\lambda_i x^*(x_i)\gamma_{z_i}=0$, and since $\Gamma(\D)$ is a linearly independent subset of $\G(\D)$ by \cite[Remark 2.8]{JimRui-22}, it follows that $x^*(x_i)\lambda_i =0$ for all $i\in\{1,\ldots,n\}$, hence $\lambda_i=0$ for all $i\in\{1,\ldots,n\}$ since $B_{X^*}$ separates the points of $X$, and thus $\gamma=0$. 

Finally, the norm $w_2^\B$ is a Bloch reasonable cross-norm since it holds the two conditions of \cite[Definition 2.5]{CabJimRui-23}:
\begin{enumerate}
	\item Given $z\in\D$ and $x\in X$, we have
$$
w_2^\B(\gamma_z\otimes x)\leq\frac{\left\|x\right\|}{1-|z|^2}=\left\|\gamma_z\right\|\left\|x\right\|.
$$
\item For $g\in\widehat{\B}(\D)$ and $x^*\in X^*$, we have 
\begin{align*}
\left|(g\otimes x^*)(\gamma)\right|&=\left|\sum_{i=1}^n\lambda_i(g\otimes x^*)(\gamma_{z_i}\otimes x_i)\right|=\left|\sum_{i=1}^n\lambda_i g'(z_i)x^*(x_i)\right|\\
&%\leq\sum_{i=1}^n\left|\lambda_i\right|\left|g'(z_i)\right|\left|x^*(x_i)\right|
\leq \left\|x^*\right\|\sum_{i=1}^n\left|\lambda_i\right|\left|g'(z_i)\right|\left\|x_i\right\|\\
&\leq\left\|x^*\right\|\left(\sum_{i=1}^{n}\left\|x_i\right\|^2\right)^{\frac{1}{2}}\left(\sum_{i=1}^{n}\left|\lambda_i\right|^{2}\left|g'(z_i)\right|^{2}\right)^{\frac{1}{2}}\\
&\leq\left\|x^*\right\|\left(\sum_{i=1}^{n}\left\|x_i\right\|^2\right)^{\frac{1}{2}}\left(\sum_{j=1}^{m}\left|\mu_j\right|^{2}\left|g'(w_j)\right|^{2}\right)^{\frac{1}{2}}\\
&\leq\left\|x^*\right\|\rho_\B(g)\left(\sum_{i=1}^n\left\|x_i\right\|^2\right)^{\frac{1}{2}}\left(\sum_{j=1}^m\frac{\left|\mu_j\right|^2}{(1-|w_j|^2)^2}\right)^{\frac{1}{2}},
\end{align*}
and taking the infimum over all representations of $\gamma$, we obtain that 
$$
\left|(g\otimes x^*)(\gamma)\right|\leq \rho_\B(g)\left\|x^*\right\|w_2^\B(\gamma).
$$
Hence $g\otimes x^*\in (\lin(\Gamma(\D))\otimes_{w_2^\B} X)^*$ with $\left\|g\otimes x^*\right\|\leq \rho_\B(g)\left\|x^*\right\|$.
\end{enumerate}
\end{proof}

We are now ready to state the announced result.

\begin{theorem} 
Let $X$ be a complex Banach space. Then $(\Gamma^{\widehat{\B}}_2(\D,X^*),\gamma^{\B}_2)$ is isometrically isomorphic to $(\lin(\Gamma(\D))\widehat{\otimes}_{w_2^\B} X)^*$, via the linear mapping $\Lambda\colon\Gamma^{\widehat{\B}}_2(\D,X^*)\to(\lin(\Gamma(\D))\widehat{\otimes}_{w_2^\B} X)^*$ defined by 
$$
\Lambda(f)(\lambda\gamma_z\otimes x)=\lambda\left\langle f'(z),x\right\rangle 
$$
for $f\in\Gamma^{\widehat{\B}}_2(\D,X^*)$, $\lambda\in\C$, $z\in\D$ and $x\in X$. Moreover, its inverse satisfies that  
$$
\left\langle(\Lambda^{-1}(\varphi))'(z),x\right\rangle=\varphi(\gamma_z\otimes x) 
$$
for $\varphi\in(\lin(\Gamma(\D))\widehat{\otimes}_{w_2^\B} X)^*$, $z\in\D$ and $x\in X$.

Moreover, on the unit ball of $\Gamma^{\widehat{\B}}_2(\D,X^*)$ the weak* topology coincides with the topology of pointwise $\sigma(X^*,X)$-convergence.
\end{theorem}

\begin{proof}
Let $f\in\Gamma^{\widehat{\B}}_2(\D,X^*)$ and let $\Lambda_0(f)\colon\lin(\Gamma(\D))\otimes X\to\mathbb{C}$ be the linear functional given by 
$$
\Lambda_0(f)(\gamma)=\sum_{i=1}^n\lambda_i\left\langle f'(z_i),x_i\right\rangle 
$$
for $\gamma=\sum_{i=1}^n\lambda_i\gamma_{z_i}\otimes x_i\in\lin(\Gamma(\D))\otimes X$ with $n\in\N$, $(\lambda_i,z_i)_{i=1}^n\in(\C\times\D)^n$ and $(x_i)_{i=1}^n\in X^n$. Let $(\mu_j,w_j)_{j=1}^m\succ (\lambda_i,z_i)_{i=1}^n$ with $m\in\N$ and $(\mu_j,w_j)_{j=1}^m\in(\C\times\D)^m$. Note that $\Lambda_0(f)\in(\lin(\Gamma(\D))\otimes_{w_2^\B} X)^*$ with 
$\left\|\Lambda_0(f)\right\|\leq \gamma^{\B}_2(f)$. Indeed, by using Theorem \ref{theo-8}, we have   
\begin{align*}
\left|\Lambda_0(f)(\gamma)\right|&=\left|\sum_{i=1}^n\lambda_i\left\langle f'(z_i), x_i\right\rangle\right|\leq\sum_{i=1}^n\left|\lambda_i\right|\left\|f'(z_i)\right\|\left\|x_i\right\|\\
&\leq \left(\sum_{i=1}^n\left\|x_i\right\|^{2}\right)^{\frac{1}{2}}\left(\sum_{i=1}^n\left|\lambda_i\right|^2\left\|f'(z_i)\right\|^2\right)^{\frac{1}{2}}\\
&\leq \gamma^{\B}_2(f)\left(\sum_{i=1}^n\left\|x_i\right\|^{2}\right)^{\frac{1}{2}}\left(\sum_{j=1}^m\frac{\left|\mu_j\right|^2}{(1-|w_j|^2)^2}\right)^{\frac{1}{2}}
\end{align*}
and taking the infimum over all the representations of $\gamma$, we deduce that $\left|\Lambda_0(f)(\gamma)\right|\leq \gamma^{\B}_2(f)w_2^\B(\gamma)$, and thus $\left\|\Lambda_0(f)\right\|\leq\gamma^{\B}_2(f)$.

%No es correcto: Since $\lin(\Gamma(\D))$ is a norm-dense linear subspace of $\G(\D)$ and $w^\B_2$ is a norm on $\G(\D)\otimes X$, then $\G(\D)\otimes X$ is a dense linear subspace of $\G(\D)\otimes_{w^\B_2} X$ and therefore also of its completion $\G(\D)\widehat{\otimes}_{w^\B_2} X$. 

Let $\lin(\Gamma(\D))\widehat{\otimes}_{w^\B_2} X$ be the completion of the space $\lin(\Gamma(\D))\otimes X$ with the norm $w^\B_2$. Hence there is a unique continuous mapping $\Lambda(f)$ from $\lin(\Gamma(\D))\widehat{\otimes}_{w^\B_2} X$ into $\mathbb{C}$ that extends $\Lambda_0(f)$. Further, $\Lambda(f)$ is linear and $\left\|\Lambda(f)\right\|=\left\|\Lambda_0(f)\right\|$.

Let $\Lambda\colon\Gamma^{\widehat{\B}}_2(\D,X^*)\to(\lin(\Gamma(\D))\widehat{\otimes}_{w^\B_2} X)^*$ be the map so defined. It is immediate that the mapping $\Lambda_0\colon\Gamma^{\widehat{\B}}_2(\D,X^*)\to(\lin(\Gamma(\D))\otimes_{w^\B_2} X)^*$ is linear. If $f\in\Gamma^{\widehat{\B}}_2(\D,X^*)$ and $\Lambda_0(f)=0$, then $\left\langle f'(z),x\right\rangle=\Lambda_0(f)(\gamma_z\otimes x)=0$ for all $z\in\D$ and $x\in X$, hence $f'(z)=0$ for all $z\in\D$ and therefore $f=0$. This proves that $\Lambda_0$ is injective. It follows easily that $\Lambda$ is also linear and injective.

To prove that $\Lambda$ is a surjective isometry, let $\varphi\in(\lin(\Gamma(\D))\widehat{\otimes}_{w^\B_2} X)^*$ and define $F_\varphi\colon\D\to X^*$ by 
$$
\left\langle F_\varphi(z),x\right\rangle=\varphi(\gamma_z\otimes x)\qquad\left(z\in \D,\; x\in X\right). 
$$
An argument similar to that in the proof of Proposition 2.4 in \cite{CabJimRui-23} assures that 
%$F_\varphi\in\H(\D,X^*)$ and 
there exists a mapping $f_\varphi\in\widehat{\B}(\D,X^*)$ with $\rho_{\B}(f_\varphi)\leq\left\|\varphi\right\|$ such that $f_\varphi'=F_\varphi$.

We now see that $f_\varphi\in\Gamma^{\widehat{\B}}_2(\D,X^*)$. Let $n,m\in\N$, $(\lambda_i,z_i)_{i=1}^n\in(\C\times\D)^n$ and $(\mu_j,w_j)_{j=1}^m\in(\C\times\D)^m$ such that $(\lambda_i,z_i)_{i=1}^n\prec (\mu_j,w_j)_{j=1}^m$. Let $(x_i)_{i=1}^n\in X^n$ and $\varepsilon>0$. For each $i\in\{1,\ldots,n\}$, there exists $x_i\in X$ with $\left\|x_i\right\|\leq 1+\varepsilon$ such that $\left\langle f_\varphi'(z_i),x_i\right\rangle=\left\|f_\varphi'(z_i)\right\|$. Clearly, the function $T\colon\C^n\to\C$ given by 
$$
T(t_1,\ldots,t_n)=\sum_{i=1}^n t_i \lambda_i\left\|f_\varphi'(z_i)\right\|,\quad\forall (t_1,\ldots,t_n)\in\C^n,
$$
is linear and continuous on $(\C^n,||\cdot||_2)$ with $\left\|T\right\|=\left(\sum_{i=1}^n\left|\lambda_i\right|^2\left\|f'_\varphi(z_i)\right\|^{2}\right)^{\frac{1}{2}}$. For all $(t_1,\ldots,t_n)\in\C^n$ with $||(t_1,\ldots,t_n)||_2\leq 1$, we have
\begin{align*}
\left|T(t_1,\ldots,t_n)\right|&=\left|\varphi\left(\sum_{i=1}^n t_i\lambda_i\gamma_{z_i}\otimes x_i\right)\right|
\leq \left\|\varphi\right\| w^\B_2\left(\sum_{i=1}^n\lambda_i\gamma_{z_i}\otimes t_ix_i\right)\\
&\leq\left\|\varphi\right\|\left(\sum_{i=1}^n\left\|t_ix_i\right\|^2\right)^{\frac{1}{2}}\left(\sum_{j=1}^m\frac{\left|\mu_j\right|^2}{(1-|w_j|^2)^2}\right)^{\frac{1}{2}}\\
&\leq(1+\varepsilon)\left\|\varphi\right\|\left(\sum_{j=1}^m\frac{\left|\mu_j\right|^2}{(1-|w_j|^2)^2}\right)^{\frac{1}{2}},
\end{align*}
therefore, 
$$
\left(\sum_{i=1}^n\left|\lambda_i\right|^2\left\|f'_\varphi(z_i)\right\|^{2}\right)^{\frac{1}{2}}
\leq(1+\varepsilon)\left\|\varphi\right\|\left(\sum_{j=1}^m\frac{\left|\mu_j\right|^2}{(1-|w_j|^2)^2}\right)^{\frac{1}{2}},
$$
and the arbitrariness of $\varepsilon$ yields  
$$
\sum_{i=1}^n\left|\lambda_i\right|^2\left\|f'_\varphi(z_i)\right\|^{2}
\leq\left\|\varphi\right\|^2\sum_{j=1}^m\frac{\left|\mu_j\right|^2}{(1-|w_j|^2)^2}.
$$
Therefore, $f_\varphi\in\Gamma^{\widehat{\B}}_2(\D,X^*)$ with $\gamma^{\B}_2(f_\varphi)\leq\left\|\varphi\right\|$ by Theorem \ref{theo-8}. 

Finally, for any $\gamma=\sum_{i=1}^n \lambda_i\gamma_{z_i}\otimes x_i\in\lin(\Gamma(\D))\otimes_{w_2^\B} X$, we get 
$$
\Lambda(f_\varphi)(\gamma) =\sum_{i=1}^n\lambda_i\left\langle f'_\varphi(z_i),x_i\right\rangle =\sum_{i=1}^n\lambda_i\varphi(\gamma_{z_i}\otimes x_i) =\varphi\left(\sum_{i=1}^n\lambda_i\gamma_{z_i}\otimes x_i\right) =\varphi(\gamma). 
$$
Hence $\Lambda(f_\varphi)=\varphi$ on a dense subspace of $\lin(\Gamma(\D))\widehat{\otimes}_{w_2^\B} X$ and, consequently, $\Lambda(f_\varphi)=\varphi$. Moreover, $\gamma^{\B}_2(f_\varphi)\leq\left\|\Lambda(f_\varphi)\right\|$, and the first assertion in the statement follows. For the second, note that  
$$
\langle(\Lambda^{-1}(\varphi))'(z),x\rangle=\langle f'_\varphi(z),x\rangle=\langle F_\varphi(z),x\rangle=\varphi(\gamma_z\otimes x)
$$
for $\varphi\in (\lin(\Gamma(\D))\widehat{\otimes}_{w_2^\B} X)^*$, $z\in \D$ and $x\in X$.

For the third, let $(f_i)_{i\in I}$ be a net in $\Gamma^{\widehat{\B}}_2(\D,X^*)$ and $f\in\Gamma^{\widehat{\B}}_2(\D,X^*)$. Assume that $(f_i)_{i\in I}\to f$ weak* in $\Gamma^{\widehat{\B}}_2(\D,X^*)$. 
%, this means that $(\Lambda(f_i))_{i\in I}\to \Lambda(f)$ weak* in $(\lin(\Gamma(\D))\widehat{\otimes}_{w^\B_2} X)^*$, that is, $(\Lambda(f_i)(\gamma))_{i\in I}\to \Lambda(f)(\gamma)$ for all $\gamma\in\lin(\Gamma(\D))\widehat{\otimes}_{w^\B_2} X$. In particular, 
This implies that 
$$
(<f'_i(z),x>)_{i\in I}=(\Lambda(f_i)(\gamma_z\otimes x))_{i\in I}\to \Lambda(f)(\gamma_z\otimes x)=<f'(z),x>
$$
for every $z\in\D$ and $x\in X$. Given $z\in\D$ and $x\in X$, we have
\begin{align*}
\left|\left\langle f_i(z)-f(z),x\right\rangle\right|&=\left|\int_{[0,z]}\left\langle f'_i(w)-f'(w),x\right\rangle\ dw\right|\\
                        &\leq |z|\max\left\{\left|\left\langle f'_i(w)-f'(w),x\right\rangle\right|\colon w\in [0,z]\right\}\\
												&=|z|\left|\left\langle f'_i(w_z)-f'(w_z),x\right\rangle\right|
\end{align*}
for all $i\in I$ and some $w_z\in [0,z]$, and thus $(\left\langle f_i(z),x\right\rangle)_{i\in I}\to \left\langle f(z),x\right\rangle$. Hence $(f_i)_{i\in I}\to f$ in the topology of pointwise $\sigma(X^*,X)$-convergence. Therefore, the identity on $\Gamma^{\widehat{\B}}_2(\D,X^*)$ is a continuous bijection from the weak* topology to the topology of pointwise $\sigma(X^*,X)$-convergence. On the unit ball, the former topology is compact and the latter is Hausdorff, and so they are equal.
\end{proof}

%%%%%%%%%%%%%%%%%%%%%%%%%%%%%%%%%%%%%%%%%%%%%%%%%%%%%%%%%%%%%%%%%%%%%%%%%%%%%%%%%%%%%%%%%%%%%%%%%%%%%%%%%%%%%%%%%%%%%%%%%%%%%%%%%%%%%%%%%%%%%%%%%%%%%%%%%%%%%%%%%%%%%%%%%%%%%%%%%%

%\section*{Statements \& Declarations} 

%\textbf{Author contributions.} All the authors have the same amount of contribution.\\B. and C.D. wrote the main manuscript text and E.F. prepared figures 1-3. All authors reviewed the manuscript.

%\textbf{Funding.} The first two authors were partially supported by Junta de Andaluc\'{\i}a grant FQM194, and by grant PID2021-122126NB-C31 funded by MCIN/AEI/ 10.13039/501100011033 and by ``ERDF A way of making Europe''.\\

%\textbf{Competing Interests.} No potential competing interest was reported by the author(s).%The authors have no relevant financial or non-financial interests to disclose.

%\textbf{Data Availability Statement.} Data sharing is not applicable to this article as no data sets were generated or analyzed during the current study.\\

%\textbf{Declarations.} Conflict of interest The authors declare that they have no competing interests.

%%%%%%%%%%%%%%%%%%%%%%%%%%%%%%%%%%%%%%%%%%%%%%%%%%%%%%%%%%%%%%%%%%%%%%%%%%%%%%%%%%%%%%%%%%%%%%%%%%%%%%%%%%%%%%%%%%%%%%%%%%%%%%%%%%%%%%%%%%%%%%%%%%%%%%%%%%%%%%%%%%%%%%%%%%%%%%%%%%

%%%%%%%%%%%%%%%%%%%%%%%%%%%%%%%%%%%%%%%%%%%%%%%%%%%%%%%%%%%%%%%%%%%%%%%%%%%%%%%%%%%%%%%%%%%%%%%%%%%%%%%%%%%%%%%%%%%%%%%%%%%%%%%%%%%

\end{document}